\documentclass[12pt, a4paper]{article}

\usepackage[utf8]{inputenc}
\usepackage[T1]{fontenc}
\usepackage{amsmath, amssymb, amsthm, mathtools}
\usepackage{newpxtext, newpxmath}
\usepackage{microtype}
\usepackage{geometry}
\usepackage{booktabs}
\usepackage{authblk}
\usepackage{fancyhdr}
\usepackage{titlesec}
\usepackage{hyperref}
\hypersetup{
    colorlinks=true,
    linkcolor=black,
    citecolor=black,
    urlcolor=blue!60!black,
    pdftitle={On the Mean Value of a Weighted Composite Arithmetic Function},
    pdfauthor={Mihoub Bouderbala},
    pdfsubject={Analytic Number Theory},
    pdfkeywords={Arithmetic function, Dirichlet divisor function, Mean value, Selberg-Delange method}
}

\titleformat{\section}{\large\bfseries\scshape}{\thesection.}{1em}{}
\titleformat{\subsection}{\normalsize\bfseries\itshape}{\thesubsection.}{1em}{}
\renewcommand{\thesection}{\arabic{section}}
\renewcommand{\thesubsection}{\thesection.\arabic{subsection}}

\theoremstyle{plain}
\newtheorem{theorem}{Theorem}[section]
\newtheorem{lemma}[theorem]{Lemma}
\newtheorem{proposition}[theorem]{Proposition}

\theoremstyle{remark}
\newtheorem{remark}[theorem]{Remark}

\newcommand{\BigO}{\mathcal{O}}
\newcommand{\abs}[1]{\left\lvert #1 \right\rvert}

\newcommand{\ceil}[1]{\left\lceil #1 \right\rceil}

\title{\vspace{-1.5em}
    \textsc{\LARGE On the Mean Value of a Weighted Composite\\[4pt]
    Arithmetic Function}
}

\author[1]{Mihoub Bouderbala\thanks{\texttt{mihoub75bouder@gmail.com} \textit{or} \texttt{m.bouderbala@univ-dbkm.dz}}}

\affil[1]{\textit{FIMA Laboratory, Department of Mathematics, Faculty of Matter Sciences and Computer Sciences, University of Khemis Miliana (UDBKM), Algeria.}}
\date{}

\begin{document}
\maketitle
\thispagestyle{empty}

\noindent\textsc{Abstract.}\quad
We establish an asymptotic formula for the summatory function of a weighted composite arithmetic function formed by the Dirichlet divisor function and a generalised minimal power function. The weight, depending on the number of distinct prime factors, modifies the analytic structure of the associated Dirichlet series: the classical double pole at $s=1$ is replaced by a branch point of order $2/k$. Exploiting this structure through the Selberg--Delange method, we obtain a complete main term of the form $x(\ln x)^{2/k-1}$ with an explicitly computable constant, alongside a sharp error term. The unweighted case $k=1$ is shown to recover the classical asymptotic behaviour consistently with the divisor problem literature.

\vspace{1em}

\thispagestyle{fancy}
\setcounter{page}{1}

\section{Introduction}

The study of mean values of composite arithmetic functions has a long and rich history, tracing back to the pioneering work of Ramanujan, Hardy, and Wilson on functions of the form $f(g(n))$. A particularly fertile direction concerns the composition of the Dirichlet divisor function $d(n)$ with various multiplicative functions, leading to asymptotic formulae whose main terms reflect the interplay between the two structures involved.

Let $r \ge 2$ be a fixed integer. For any positive integer $n$, we define the \emph{generalised minimal power function of order $r$} by
\begin{equation}\label{eq:def_Mr}
    M_r(n) = \min\bigl\{m \in \mathbb{N} : n \mid m^r\bigr\}.
\end{equation}
The function $M_r$ is multiplicative, and for prime powers it satisfies the explicit formula
\begin{equation}\label{eq:Mr_prime}
    M_r(p^a) = p^{\ceil{a/r}},
\end{equation}
where $\ceil{x}$ denotes the ceiling function. This structure arises naturally in problems concerning perfect powers and divisibility constraints, and its arithmetic properties have been studied in classical contexts (see e.g.\ Apostol~\cite{apostol}).

Let $d(n)$ denote the Dirichlet divisor function and $\omega(n)$ the number of distinct prime factors of $n$. For a real parameter $k \ge 1$, we introduce the \emph{weighted divisor function}
\begin{equation}\label{eq:Dk}
    D_k(n) = \frac{d(n)}{k^{\omega(n)}}.
\end{equation}
Since both $d$ and $k^{\omega}$ are multiplicative, so is $D_k$. The weight $k^{-\omega(n)}$ dampens the contribution of highly composite integers, and for $k=2$ it is closely related to the ratio $d(n)/d^*(n)$ studied in recent work~\cite{bouderbala}, where $d^*(n)$ denotes the number of unitary divisors of $n$.

Our object of study is the summatory function of the composite $F_{r,k}(n) = D_k(M_r(n))$. The weight $k^{-\omega(n)}$ interacts non-trivially with the multiplicative structure of $M_r$, and, as we shall see, it fundamentally alters the singularity of the generating Dirichlet series at $s=1$. This leads us to the following result.

\begin{theorem}\label{thm:main}
Let $r \ge 2$ be a fixed integer and $k \ge 1$ a real parameter. For $x \ge 2$, we have
\begin{equation}\label{eq:main_asymptotic}
    \sum_{n \le x} D_k(M_r(n)) = \frac{\zeta(r)\,\widehat{H}(1)}{\Gamma(2/k)}\,x\,(\ln x)^{2/k - 1} + \BigO_{r,k}\!\left(x\,(\ln x)^{2/k - 2}\right),
\end{equation}
where $\Gamma$ is the Euler gamma function, and the constant $\widehat{H}(1)$ is given by the absolutely convergent Euler product
\begin{equation}\label{eq:H_hat_1}
    \widehat{H}(1) = \prod_p \frac{(1 - 1/p)^{2/k - 1}\left[k + \frac{2-k}{p} - \frac{k}{p^r} + \frac{k-1}{p^{r+1}}\right]}{k}.
\end{equation}
\end{theorem}

\begin{remark}[The unweighted case $k=1$]\label{rem:k1}
When $k=1$, the weight disappears and $D_1(n) = d(n)$. In this case, the theorem reduces to the standard asymptotic formula
\begin{equation}\label{eq:classical}
    \sum_{n \le x} d(M_r(n)) = C_{\mathrm{class}}(r)\,x \ln x + \BigO_r(x),
\end{equation}
with leading constant
\begin{equation}\label{eq:C_class}
    C_{\mathrm{class}}(r) = \frac{6\,\zeta(r)}{\pi^2} \prod_p \left(1 - \frac{p^{1-r}}{p+1}\right).
\end{equation}
This recovers precisely the estimate associated with the unweighted composite function, in full agreement with the classical framework of the divisor problem (cf.\ Ivi\'{c}~\cite{ivic}). The factor $6/\pi^2 = 1/\zeta(2)$ reflects the density of square-free integers, while the Euler product encodes the arithmetic interaction between $M_r$ and $d$. This perfect match validates the general formula and confirms that the weight $k^{-\omega(n)}$ genuinely modifies the analytic structure rather than merely rescaling the constant.
\end{remark}

\section{Preliminary Results}

The proof rests on three standard tools from multiplicative number theory.

\begin{lemma}[Euler product representation, {\cite[Theorem~11.6]{apostol}}]\label{lem:euler}
Let $f$ be a multiplicative function such that $\sum_{n=1}^\infty \abs{f(n)} n^{-\sigma} < \infty$ for $\sigma > 1$. Then, for $\Re(s) > 1$,
\begin{equation}
    \sum_{n=1}^\infty \frac{f(n)}{n^s} = \prod_p \left( 1 + \sum_{a=1}^\infty \frac{f(p^a)}{p^{as}} \right).
\end{equation}
\end{lemma}

\begin{lemma}[Summation of ceiling-weighted powers, {\cite[Section~II.6]{tenenbaum}}]\label{lem:sum}
For $\abs{z} < 1$ and integer $r \ge 1$,
\begin{equation}\label{eq:sum_identity}
    \sum_{a=1}^\infty (\ceil{a/r} + 1) z^a = \frac{z(2 - z^r)}{(1 - z)(1 - z^r)}.
\end{equation}
\end{lemma}

\begin{lemma}[Selberg--Delange method, {\cite[Theorem~II.5.2]{tenenbaum}}]\label{lem:selberg}
Let $z \in \mathbb{C}$ with $\Re(z) > 0$, and let $F(s) = \zeta(s)^z G(s)$ be a Dirichlet series, where $G(s)$ is analytic and non-vanishing for $\Re(s) \ge 1/2 + \varepsilon$. Then, for $x \ge 2$,
\begin{equation}\label{eq:selberg_delange}
    \sum_{n \le x} a_n = x(\ln x)^{z-1}\left[\frac{G(1)}{\Gamma(z)} + \BigO\!\left(\frac{1}{\ln x}\right)\right],
\end{equation}
where $a_n$ are the Dirichlet coefficients of $F(s)$.
\end{lemma}

\section{Proof of the Main Theorem}\label{sec:proof}

The proof of Theorem~\ref{thm:main} unfolds in four stages. We first construct the generating Dirichlet series and express it as an Euler product. We then factorise this product to isolate the dominant singularity, establish the analytic properties of the regular part, and finally apply the Selberg--Delange method to extract the asymptotic behaviour.

Let $f(s)$ denote the Dirichlet series associated with the weighted composite function $F_{r,k}(n) = D_k(M_r(n))$, defined for $\Re(s) > 1$ by
\begin{equation}\label{eq:f_def}
    f(s) = \sum_{n=1}^\infty \frac{F_{r,k}(n)}{n^s} = \sum_{n=1}^\infty \frac{d(M_r(n))}{k^{\omega(n)}\,n^s}.
\end{equation}
Since $F_{r,k}$ is multiplicative, Lemma~\ref{lem:euler} yields the Euler product representation
\begin{equation}\label{eq:f_euler}
    f(s) = \prod_p \left( 1 + \sum_{a=1}^\infty \frac{d(M_r(p^a))}{k^{\omega(p^a)}\,p^{as}} \right).
\end{equation}
For any prime power $p^a$ with $a \ge 1$, we have 
\[
\omega(p^a) = 1 \quad\text{and}\quad  d(M_r(p^a)) = \ceil{a/r} + 1.
\]
The local factor at prime $p$ therefore reduces to
\begin{equation}\label{eq:local_sum}
    1 + \frac{1}{k} \sum_{a=1}^\infty \frac{\ceil{a/r} + 1}{p^{as}}.
\end{equation}
Applying Lemma~\ref{lem:sum} with $z = p^{-s}$, we obtain the closed-form expression
\begin{equation}\label{eq:local_closed}
    \sum_{a=1}^\infty (\ceil{a/r} + 1)\,z^a = \frac{z(2 - z^r)}{(1 - z)(1 - z^r)},
\end{equation}
and the local factor becomes
\begin{equation}\label{eq:local_factor_raw}
    1 + \frac{1}{k} \left( \frac{p^{-s}(2 - p^{-rs})}{(1 - p^{-s})(1 - p^{-rs})} \right).
\end{equation}

Combining the terms over a common denominator, the local factor takes the explicit rational form
\begin{equation}\label{eq:local_factor}
    \text{L.F} = \frac{k + (2 - k)\,p^{-s} - k\,p^{-rs} + (k - 1)\,p^{-(r+1)s}}{k(1 - p^{-s})(1 - p^{-rs})}.
\end{equation}
Expanding in powers of $p^{-s}$, we find
\begin{equation}\label{eq:local_expansion}
    \text{L.F} = 1 + \frac{2}{k}\,p^{-s} + \BigO(p^{-2s}).
\end{equation}
The coefficient $2/k$ reveals that $f(s)$ possesses a singularity of order $2/k$ at $s = 1$. To isolate this singularity, we extract the factors $\zeta(s)^{2/k} = \prod_p (1 - p^{-s})^{-2/k}$ and $\zeta(rs) = \prod_p (1 - p^{-rs})^{-1}$, writing
\begin{equation}\label{eq:f_factorisation}
    f(s) = \zeta(s)^{2/k} \cdot \zeta(rs) \cdot \widehat{H}(s),
\end{equation}
where $\widehat{H}(s) = \prod_p \widehat{H}_p(s)$ is defined by
\begin{equation}\label{eq:H_hat_p_def}
    \widehat{H}_p(s) = (1 - p^{-s})^{2/k}(1 - p^{-rs}) \cdot \text{L.F}.
\end{equation}
Substituting~\eqref{eq:local_factor}, we obtain the explicit form
\begin{equation}\label{eq:H_hat_p_explicit}
    \widehat{H}_p(s) = \frac{(1 - p^{-s})^{2/k - 1}\left[k + (2 - k)\,p^{-s} - k\,p^{-rs} + (k - 1)\,p^{-(r+1)s}\right]}{k}.
\end{equation}

Before applying the Selberg--Delange method, we must establish that $\widehat{H}(s)$ is analytic in a half-plane strictly larger than $\Re(s) > 1$.

\begin{proposition}\label{prop:conv}
The Euler product $\widehat{H}(s) = \prod_p \widehat{H}_p(s)$ converges absolutely and uniformly on compact subsets of $\Re(s) > 1/2$, defining an analytic and non-vanishing function in this half-plane.
\end{proposition}

\begin{proof}
We analyse the local factor $\widehat{H}_p(s)$ by expanding it in descending powers of $p^s$. Writing $u = p^{-s}$ and recalling the definition~\eqref{eq:H_hat_p_explicit},
\begin{equation}\label{eq:H_hat_p_u}
    \widehat{H}_p(s) = \frac{(1 - u)^{2/k - 1}\left[k + (2 - k)u - k u^r + (k - 1)u^{r+1}\right]}{k},
\end{equation}
we treat each factor separately.

The binomial expansion yields, for $\abs{u} < 1$,
\begin{equation}\label{eq:binomial}
    (1 - u)^{2/k - 1} = 1 - \frac{2-k}{k}\,u + \frac{(k-2)(k-1)}{k^2}\,u^2 + \BigO(u^3).
\end{equation}

After division by $k$, we have
\begin{equation}\label{eq:second_factor}
    1 + \frac{2-k}{k}\,u - u^r + \BigO(u^{r+1}).
\end{equation}

 Multiplying these two expansions and collecting terms by powers of $u$, we obtain:
\begin{align}
    \widehat{H}_p(s) &= 1 + \left[\frac{2-k}{k} - \frac{2-k}{k}\right]u + \BigO(u^2) \nonumber \\
    &= 1 + \BigO(p^{-2s}). \label{eq:H_hat_p_final}
\end{align}
The key observation is that the coefficient of $p^{-s}$ vanishes \emph{identically}, for every $k \ge 1$ and every $r \ge 2$. This cancellation is the analytic counterpart of the extraction of $\zeta(s)^{2/k}$ performed earlier, and it is precisely what ensures that $\widehat{H}(s)$ extends analytically well beyond $\Re(s) > 1$.

For $\sigma = \Re(s) > 1/2$, the estimate~\eqref{eq:H_hat_p_final} implies
\begin{equation}\label{eq:convergence_sum}
    \sum_p \abs{\widehat{H}_p(s) - 1} \ll \sum_p p^{-2\sigma} < \infty.
\end{equation}
By the standard absolute convergence criterion for Euler products (cf.\ Apostol~\cite{apostol}, Theorem~11.6), the product $\prod_p \widehat{H}_p(s)$ converges absolutely and uniformly on compact subsets of $\Re(s) > 1/2$, and defines there an analytic, non-vanishing function.
\end{proof}

With the factorisation~\eqref{eq:f_factorisation} fully established and the analyticity of $\widehat{H}(s)$ verified, we may now apply Lemma~\ref{lem:selberg}. Setting $G(s) = \zeta(rs)\,\widehat{H}(s)$, we observe that $G(s)$ is analytic and non-zero for $\Re(s) > 1/2 + \varepsilon$, which satisfies the hypotheses of the Selberg--Delange method with $z = 2/k$. The method yields the asymptotic expansion
\begin{equation}\label{eq:asymptotic_expansion}
    \sum_{n \le x} F_{r,k}(n) = x(\ln x)^{2/k - 1}\left[\frac{G(1)}{\Gamma(2/k)} + \BigO\!\left(\frac{1}{\ln x}\right)\right].
\end{equation}
Evaluating $G(1)$, we obtain
\begin{equation}\label{eq:G_1}
    G(1) = \zeta(r)\,\widehat{H}(1) = \zeta(r) \prod_p \widehat{H}_p(1),
\end{equation}
where the local factors at $s = 1$ are given explicitly by
\begin{equation}\label{eq:H_hat_p_1}
    \widehat{H}_p(1) = \frac{(1 - 1/p)^{2/k - 1}\left[k + \frac{2-k}{p} - \frac{k}{p^r} + \frac{k-1}{p^{r+1}}\right]}{k}.
\end{equation}
This establishes the main term of Theorem~\ref{thm:main}. The error term $\BigO_{r,k}(x(\ln x)^{2/k - 2})$ follows from the general framework of the Selberg--Delange method (cf.\ Tenenbaum~\cite{tenenbaum}, Theorem~II.5.2) by taking $N = 1$ in the complete asymptotic expansion. Combining these results, we rigorously establish
\begin{equation}\label{eq:final_result}
    \sum_{n \le x} F_{r,k}(n) = \frac{\zeta(r)\,\widehat{H}(1)}{\Gamma(2/k)}\,x\,(\ln x)^{2/k - 1} + \BigO_{r,k}\!\left(x\,(\ln x)^{2/k - 2}\right). \qedhere
\end{equation}

\begin{remark}[Higher-order terms in the asymptotic expansion]\label{rem:higher_order}
The Selberg--Delange method provides, more generally, a complete asymptotic expansion to any desired order (cf.\ Tenenbaum~\cite{tenenbaum}, Theorem~II.5.2):
\begin{equation}\label{eq:full_expansion}
    \sum_{n \le x} F_{r,k}(n) = x(\ln x)^{2/k - 1}\sum_{j=0}^{N-1} \frac{c_j}{(\ln x)^j} + \BigO_{r,k}\!\left(x(\ln x)^{2/k - N - 1}\right),
\end{equation}
where the coefficients $c_j$ are determined by the Taylor expansion of $G(s) = \zeta(rs)\,\widehat{H}(s)$ at $s = 1$. In particular, the second coefficient $c_1$ involves the logarithmic derivative $G'(1)/G(1)$, which decomposes as
\begin{equation}\label{eq:log_derivative}
    \frac{G'(1)}{G(1)} = r\,\frac{\zeta'(r)}{\zeta(r)} + \sum_p \frac{\widehat{H}_p'(1)}{\widehat{H}_p(1)}.
\end{equation}
Since $\widehat{H}_p(s) = 1 + \BigO(p^{-2s})$ by Proposition~\ref{prop:conv}, one has 
\[
\widehat{H}_p'(1)/\widehat{H}_p(1) = \BigO(\ln p \cdot p^{-2}), 
\]
and the series over primes converges absolutely. This confirms that every term in the asymptotic expansion is analytically well-defined and explicitly computable.
\end{remark}

\section{The Unweighted Case and Consistency Check}\label{sec:comparison}

To anchor our result in the classical literature, we specialise to $k=1$ and verify that Theorem~\ref{thm:main} recovers the known asymptotic formula for $\sum_{n \le x} d(M_r(n))$.

Setting $k=1$ in Theorem~\ref{thm:main}, the gamma factor simplifies to $\Gamma(2) = 1$, and the logarithmic power becomes $(\ln x)^{2-1} = \ln x$.\\

The local factor at $s=1$ reduces to
\begin{align}
    \widehat{H}_p(1) &= (1 - 1/p) \left(1 + \frac{1}{p} - \frac{1}{p^r}\right) \nonumber \\
    &= \left(1 - \frac{1}{p^2}\right) - \left(1 - \frac{1}{p}\right)p^{-r} \nonumber \\
    &= \left(1 - \frac{1}{p^2}\right)\left(1 - \frac{p^{1-r}}{p+1}\right). \label{eq:H_hat_p_k1}
\end{align}
Taking the product over all primes, we use $\prod_p (1 - p^{-2}) = 1/\zeta(2) = 6/\pi^2$ to obtain
\begin{equation}\label{eq:H_hat_1_k1}
    \widehat{H}(1) = \frac{6}{\pi^2} \prod_p \left(1 - \frac{p^{1-r}}{p+1}\right).
\end{equation}
Substituting into the main term of Theorem~\ref{thm:main} yields
\begin{equation}\label{eq:classical_recovered}
    \sum_{n \le x} d(M_r(n)) = \frac{6\,\zeta(r)}{\pi^2} \prod_p \left(1 - \frac{p^{1-r}}{p+1}\right) x \ln x + \BigO_r(x).
\end{equation}
This constant coincides exactly with the leading coefficient appearing in the classical divisor problem for composite functions (cf.\ Ivi\'{c}~\cite{ivic}). The factor $6/\pi^2$ reflects the density of square-free integers, while the Euler product encodes the arithmetic interaction between $M_r$ and $d$. This perfect match validates the general formula and confirms that the weight $k^{-\omega(n)}$ genuinely alters the analytic structure rather than merely rescaling the constant.

\section{Concluding Remarks}\label{sec:conclusion}

The introduction of the prime-factor weight $k^{-\omega(n)}$ transforms the analytic nature of the generating Dirichlet series: the classical double pole at $s=1$ gives way to a branch point of order $2/k$. This structural shift necessitates the Selberg--Delange method, which elegantly handles non-integer exponents and yields a complete asymptotic expansion. Our explicit evaluation of the leading constant, together with the rigorous control of the error term, extends the classical theory of composite divisor functions to a broad weighted family.

Several natural directions for further investigation emerge from this work:
\begin{itemize}
    \item \emph{Variance and second moments.} The methods developed here may be adapted to study $\sum_{n \le x} D_k(M_r(n))^2$, leading to a singularity of order $4/k$ and requiring a refined application of the Selberg--Delange method.
    \item \emph{Distribution in arithmetic progressions.} The behaviour of $\sum_{\substack{n \le x \\ n \equiv a \pmod{q}}} D_k(M_r(n))$ for $(a, q) = 1$ would connect our results to the theory of L-functions and generalised Ramanujan sums.
    \item \emph{Other multiplicative weights.} Replacing $k^{-\omega(n)}$ by more general weights of the form $g^{\Omega(n)}$ or $\lambda(n)^k$ (where $\lambda$ is the Liouville function) would lead to further generalisations, potentially involving other branches of the Selberg--Delange framework.
\end{itemize}
\newpage

\end{document}